\documentclass[12pt]{article}
\usepackage{authblk}
\usepackage{amssymb}
\usepackage{amsmath,amsthm}

\providecommand{\U}[1]{\protect\rule{.1in}{.1in}}
\makeatletter
\def\@seccntformat#1{\csname the#1\endcsname.\quad}
\makeatother

\newtheorem{theorem}{Theorem}[section]

\newtheorem{corollary}{Corollary}[theorem]

\newtheorem{definition}{\noindent Definition}
\newtheorem{example}{Example}

\newtheorem{lemma}[theorem]{Lemma}

\newtheorem{proposition}[theorem]{Proposition}
\newtheorem{remark}{Remark}

\renewenvironment{proof}[1][Proof]{\noindent\textbf{#1.} }{\ \rule{0.5em}{0.5em}}
\begin{document}
\title{{\Large \textbf{ Novel Excitation of local fractional dynamics}}}
\author[1]{ Dhurjati Prasad Datta\thanks{Corresponding author; email:dp${_-}$datta@yahoo.com}} 
\author[2] {Soma Sarkar}
\author[3] {Santanu Raut} 
\affil[1] {IUCAA  Centre for Astronomy Research and Development, University of North Bengal, Siliguri, Pin: 734013, India}
\affil[2] {Sunity Academy, Coochbehar, Pin:736170, India}
\affil[3] {Department of Mathematics, Mathabhanga College, Coochbehar, Pin:736146, India}
\date{}
\maketitle
\begin{abstract}  
{\small 
The question of a possible excitation and emergence of fractional type dynamics, directly from a conventional integral order dynamics, as a continuous transition or deformation, is of significant current interest. 
Although there have been  lot of activities in nonlinear, fractional or not, dynamical systems, the above question appears yet to be addressed systematically in the current literature of complex systems studies. The present work may be considered to be a step forward in this direction.  Based on a novel concept of asymptotic duality structure, we present here an extended analytical framework that would provide a scenario for realizing the above stated continuous deformation of integral order dynamics to a local fractional order dynamics on a fractal and fractional space realized in an asymptotic temporal or spatial scale. The related concepts of self dual and strictly dual asymptotics are introduced and there relevance in connection with smooth and nonsmooth deformation of the original real line are pointed out. The  relationship of the duality structure and renormalization group is examined. The ordinary derivation operator is shown to be invariant under this duality enabled renormalization group transformation, leading thereby to a {\em natural} realization of local fractional type derivative in a fractal space. As an application we discuss linear wave equation in one and two dimensions and show how the underlying integral order wave equation could be deformed and renormalized suitably to yield meaningful results for vibration of a fractal string or wave propagation in a region with fractal boundary.
 } 
\end{abstract}

\begin{center}



\end{center}
\baselineskip=14.8pt

\section{Introduction}

There have been  considerable interests recently in modeling complex real world phenomena by  local fractional dynamics(LFD). The basic motivation of LFD arose from the necessity to remedy the problems of nonlocality and lack of translation invariance of the classical definitions of fractional  derivatives involving nonlocal integro-differential operators with generally power law kernels, as in the definition of Riemann-Liouville fractional derivative \cite{kol1}. Over the past few decades it became increasingly more and more evident that conventional dynamical formalism based on ${\bf R}^n$ calculus, and the associated  theory of (nonlinear) integral order differential equations  is fundamentally deficient \cite{kol1,kol2, west,tara1, tsallis} in analyzing and deriving right scaling properties of emergent (i.e. evolutionary) complex structures in various natural and other real world problems, such as the origin and proliferation of biological  structures  \cite{west}, financial modeling \cite{mandel1}, turbulence \cite{fris}, large scale structure formation in cosmology \cite{cos}, meteorological predictions \cite{meteor}, to name a few. To counter this challenging problem of more realistic modeling real world phenomena, applications of fractional calculus techniques in the form of fractional dynamics \cite{west, tara1} have gained significance over the later half of the past century. However, the nonlocal integro-differential operators of fractional theories have, nevertheless, paused several conceptual and theoretical problems. One of the major problems of such operators is that the conventional definitions of fractional derivatives do not enjoy simple physical or geometrical meaning in the sense of rate measurer or existence of tangents as in the case of integer order derivative. On the other hand, although nonlocality and power law kernels may be exploited to model memory effects in various complex collective systems, this makes conventional fractional derivatives unsuitable to study local singularity structure of fractal/multifractal signals \cite{kol1}. 

To remove these limitations the theory of LFD was originally formulated by kolwankar and Gangal \cite{kol1,kol2} in the context of Riemann-Liouville derivative, that was later reinterpreted and simplified by various authors subsequently \cite{adda, wen, prod1,prod2}. In the later formulations, authors formulated LFD as the limit of a differential quotient when the increment in independent variable  is assumed to be a power law $h^{\alpha}, \ 0<\alpha<1$, and established its equivalence with the original formulation and also discussed its utility in the study of Holder continuous functions\cite{adda, wen, prod1,prod2} . The approach receives wide applications in various scaling phenomena having predominant fractal support \cite{bel1, bel2, bel3}.
It is also recognized recently that, although formulated originally on the real line $\bf R$ \cite{kol1}, the study of LFD should be restricted to events/processes having fractal support \cite{kol2, cress, prod2}, since otherwise the local fractional derivative of a continuous function defined on $\bf R$  vanishes almost everywhere (see criticism by Tarasov \cite{tarasov1,tarasov2, cress}). In conformity with above realizations, the theory of LFD is subsequently extended successfully over fractal sets (e.g. Cantor sets or continuous fractal curves) \cite{parv1,parv2, bel1, bel2, bel3}, with the underlying fractal set being parametrized either by the associated integral staircase type monotonic function $S^s_F(x)$ \cite{parv1,parv2, bel1} or simply by a power law $x^s$ \cite{bel2,bel3}, where $x\in F$ and $s$ being the corresponding Hausdorff dimension of the fractal set $F$. An analysis on a fractal set $F\subset \bf R^2$ (having topological dimension either 0 or 1) based on LFD has the conceptual advantage that the associated derivation operator $D_{ F}$ is linear and satisfies most of  the algebraic properties of ordinary derivation on $\bf R$, and so allows straightforward physical/geometrical interpretations. A conjugacy between the integral order calculus and the LFD based fractal calculus is also established \cite{parv2}. 

During the ensuing period, another approach to analysis on fractal sets is being pursued independently by the present authors, that parallels closely the spirit of LFD, but with a larger framework aiming to realize the fractional dynamics on a fractal set as a continuous deformation of the ordinary integral order calculus (or dynamics) in an asymptotic spatio-temporal limit (either as $x\rightarrow 0$, or $t\rightarrow \infty$, or both simultaneously) \cite{dpr1,dpr2, dpr3,  dp, dpp, dss}. In the framework of conventional LFD, definitions of fractional derivatives are introduced in an ad hoc manner with specific applications in mind, and the frameworks of integral order and fractional calculus and dynamics remain as unconnected theories without having any continuity in the sense of an asymptotic deformation as conceived and developed in the present  paper or in any other analogous sense. The problem of asymptotic emergence of chaos in a driven nonlinear oscillation through period doubling route, for instance, leading to the break down of the original integral order governing equation on chaotic attractor, may be considered as one of the motivations of the above philosophy. A LFD type driving equation on the chaotic attractor is expected to offer new insights into the true nature of chaotic dynamics. The present paper may be treated as the continuation and further elaboration on the concepts and results reported in \cite{dss}. 

By asymptotic deformation, we mean a natural deformation in the family of Cauchy null real sequences ${\cal S}_0$ via a sort of renormalization group (RG) transformation $\cal T$ so that a special class of such analytically null sequences $\{a_n\}\in {{\cal S}_0}$, acquire a finite (non-null), real ultrametric value $\{a_n\}\mapsto {\cal T}(\{a_n\})$ so that $|{\cal T}(\{a_n\})|:=v(\{a_n\}) \ (>0)$,  $v$ denotes a {\em discrete}  ultrametric absolute value \cite{katok}: $v:{\cal S}_0\mapsto \bf R^+$, called {\em asymptotic renormalized valuation}. The renormalized null set ${\tilde {\cal S}_0}$, as a consequence, is realized as a totally disconnected ultrametric space, equipped with  this non-null, real valued norm $v$, having a countable set of distinct, {\em non-dense or densely defined}, values that are  induced by the renormalized transformation $\cal T$. Because  of the continuity of  the ultrametric norm,  the renormalized null set ${\tilde {\cal S}_0}$  is  then  mapped, in general, into a  connected, fractal  set (curve) ${\cal O} =v({\tilde {\cal S}_0})$, thereby inducing a {\em continuous deformation} of a linear neighbourhood of $0\in \bf R$ into a deformed neighbourhood ${\cal O}\subset \cal R$, of the associated deformed real line $\cal R$.  We show that the RG induced fractal curve $\cal O$ in the deformed neighbourhood of 0 would generally have  the structure of a piece-wise smooth or increasing or strictly increasing devil's staircase type singular function that is parametrized by a suitably rescaled real {\em scaling}       parameter $\xi=\log x/\delta$ (where $0<\delta (x)<x, \ \delta(x) \rightarrow 0$ as $x\rightarrow 0$ ). {\em Depending on the density of the ultrmatric value set}, the map is either (i) a piece-wise smooth or (ii)    a nondereasing , Lebesgue-Cantor's singular function,  or may even be  a (iii)  strictly increasing, Lebesgue-Cantor's type integral function \cite{parv2} corresponding to a no-where differentiable fractal curve. The judicious choice of {\em right} (i.e. most appropriate) valuation $v$ and the associated fractal singular function should generally be dictated by the actual physical or dynamical context asking for a formulation of fractal/fractional type dynamics.

Formation of such an extended nonlinear structure $\cal O$ is facilitated by the existence of a duality relationship in the renormalized null sequences and their associated ultrametric values. In the next section Sec.2, we formulate this concept of asymptotic duality and show how this concept leads, at the first level, to a non-archimedean extension \cite{katok} of the ordinary real line $\bf R$ into a structured field $\bf ^*R$ by extending each point $r\in \bf R$, interpreted as a limiting closed interval $[r-\delta,r+\delta], \ \delta\rightarrow 0$,  into an asymptotic renormalized set $^*r:={\bf ^*O}(r)$, respecting the translation symmetry viz., ${\bf ^*O}(r)(r)=r+{\bf ^*O}$, where ${\bf ^*O}:={\bf ^*O}(0)$. The asymptotic valuation $v$ acting in the asymptotic neighbourhood of each point $^*r\in {\bf ^*R}$ then induces a sort of deformation in the original real line $\bf R$ into a deformed real line $\cal  R$ admitting nontrivial geometric structures depending on the precise nature of the asymptotic valuation. It transpires that the renormalized null sequences in  ${\bf ^*O}$ may be realized either as {\em self dual, weakly self dual} or {\em strictly dual} asymptotics and, hence,  the associated deformed set, ${\cal O}=v({\bf ^*O})$, would accordingly have the structure of {\em a line segment, or a collection of finitely many, but continuous,  broken line segments, or a fractal curve}, respectively in that order. 

We next show that the ordinary derivation in $\bf R$ is invariant under the RG transformation $\cal T$, and consequently the ordinary derivation extends naturally to a LFD type derivation over the deformed (nonlinear) set $\cal R$ relative to the renormalized scaling variable $v(\xi)$, when the  continuously differentiable function space over $\bf R$ is extended {\em self-similarly} over the extended space $\cal R$. The presented analytic structures and associated RG transformation now offers a novel mechanism of continuously deforming the original integral order calculus and dynamics into the fractional dynamics. In other words, {\em the formalism of LFD may be said to have been  derived    from the integral order dynamics, from first principle, in the present duality enhanced framework of ordinary analysis}. 

As an application of this extended nonlinear formalism, we study linear wave equation in one and two dimensions. The invariance of the wave equation under asymptotic RG transformation is exploited to transform the equation into an LFD type wave equation in an asymptotic neighbourhood of $(t,x,y)$, when one invokes and justifies a natural spatio-temporal coherence in the respective spatial and temporal duality relationships \cite{dss}. The  deformed wave equation in 1 dimension is used to model standing fractal wave vibration of a fractal string.
 The linear dispersion relation $\omega=ck$ of the 1d wave equation, for instance,  is transformed into a fractal dispersion of the form $\omega_f=c_fk_f$, where the fractal wave number $k_f=k_f(k)$ inherits the shape of the underlying fractal string \cite{fracfreq}.

Next, the deformed wave equation in 2 dimension is considered for modeling  vibration of a 2d membrane with fractal boundary. Specifically, we study the wave equation with fractal  boundary when the boundary curve is a quadratic Koch type II fractal curve following Ref.\cite{su}. In the cited paper, the authors used a separation of variable method  on a sequence of boundary value problems with polygonal boundaries, approximating the original problem, successively, that should to converge to actual fractal boundary problem, thereby resulting in the desired solution of original problem. However, the solution reported in the paper \cite{su} was in the form of a lacunary series and represented a continuous,  no-where differentiable function, reflective of the fractal nature of the boundary. However, the reported solution is deficient in many ways, predominant being the fact that the considered approximate polygonal boundaries apparently fail to yield the intended fractal boundary in the limit, and the lacunary series can not be justifiably claimed to  be the correct solution of the problem. Here, we show that the duality structure provides one with an interesting avenue, not only in removing the above deficiency, but also to offer a more realistic modeling of the two dimensional wave vibration by a LFD type fractal wave equation in appropriate fractal scaling variables.  Implementing judiciously and coherently \cite{dss} the renormalized deformation technique across space and time coordinates,  the fractal  wave equation is shown to yield, by simple separation of variable method, a continuous solution that is also twice continuously differentiable in the extended sense of fractal differentiability \cite{dss} analogous to LFD \cite{parv2}. The solution is, however, indeed a fractal function through the argument scaling variables which are essentially integral mass functions corresponding to the fractal boundary.

The paper is organized as follows. In Sec.2 we present the concept of asymptotic duality structure and associated analytic formalism in detail \cite{dss}, introducing and proving several new concepts. A particularly important fact is the distinction between self dual and strictly dual asymptotics and there role in determining the smooth or nonsmooth structures in the vicinity of the real line via asymptotic duality transformations (sec.2.2). The relationship between duality structure and RG is  discussed in Sec.2.3. In Sec.3.1, the renormalibility of the ordinary derivative is proved, which is later extended naturally to introduce a LFD type derivation for fractal space. The problem of wave equation in one and two dimensions are studied in subsequent section 3.2 and 3.3 respectively. In Sec.4, we summarize the paper with some concluding remarks.

\section{Asymptotic duality}

Here, we review the asymptotic duality structure in the real number system \cite{dss} and also present some new results. This will be utilized later to construct  well defined $C^2$ solutions in appropriate renormalized (deformed) asymptotic variables for a wave equation with fractal boundaries.

Let us begin by recalling that every point $r$ in the real line $\bf R$ essentially represents the equivalence class of the  Cauchy sequences , not only  of the rational, but more generally, of real numbers, converging to the chosen number $r$. For simplicity, let us first fix $r$ to zero, i.e. $r=0$. We then have the Cauchy null sequences of real numbers, denoted ${\cal S}_0$.

Let us choose and fix a null sequence $a=\{a_n: a_n\in \ {\bf R}\}, \ a_n>0$,   said to be a {\em scale}. Relative to this chosen scale, the null set ${\cal S}_0$ gets  renormalized  in the form ${\tilde{\cal S}}^a_{0}:={\cal S}^a_{0}\cup {\cal S}^0_0$
where, ${\cal S}^a_{0}=\{A^{\pm}|\ |A^{\pm}|=\{a_{n}\times a_{n}^{\mp b^{\pm}_{n}}\}\}$ (the symbol $\times$ denotes ordinary product operation in $\bf R$), when $b^{\pm}_{n}> 0$ are  two non-null Cauchy sequences such that the renormalized sequences $A^{\pm} \in {\cal S}_0$ must also belong to null set, and ${\cal S}^0_0=\{A^{\pm}_0| \ |A^{\pm}_0|=\{a_{n}\times a_{n}^{\mp b^{\pm}_{n}}\}, \ A^{\pm}_0 \in {\cal S}_0\}$, when  $b^{\pm}_n$ is null or divergent (either to $\infty$ or oscillating).  As a consequence, the exponentiated sequences
$\{b_n^\pm\}$ belong to the set of all (Cauchy or not) sequences of real numbers. It follows, in fact, that
 
\begin{proposition} {\rm \cite{dss}}
 ${\tilde {\cal S}}^a_0={\cal S}_0$. 
\end{proposition}

\begin{proof}
Clearly, ${\tilde {\cal S}}^a_0\subset{\cal S}_0$, by definition. Conversely, given $\bar a=\{\bar a_n\} \in {\cal S}_0$,
there exists $\bar A\in {\cal S}^a_{0}$ such that either (i) $|\bar A|=\{a_n\times a_n^{-v_n(\bar a)}\}, \ v_n(\bar a)<1$ for $n>N$, or 
(ii) $|\bar A|=\{a_n\times a_n^{+v_n(\bar a)}\}, \ v_n(\bar a)\geq 1$, where $v_n(\bar a)= { |\log_{a_n^{-1}}|{\bar a}_n/a_n| \ |}>0$ ,
when $\bar a_n\neq 0$ for a sufficiently large n, and $\lim v_n(\bar a) $ as
$n\rightarrow \infty$ exists finitely. Otherwise, that is, when $\lim  v_n(\bar a) $ as
$n\rightarrow \infty$ does not exists (either divergent to $\infty$ or oscillating),  we set the exponent in (ii) equal to $\infty$.
As a consequence, we have the reverse inequality, $\tilde{\cal S}^a_{0}\supset {\cal S}_0$ for any $a\in {\cal S}_0, \ a\neq 0$.
\end{proof}

\begin{remark}{\rm 
Once a null sequence $a$ is identified as {\em scale}, the set of null sequences ${\cal S}_0$ is rearranged (renormalized) with redressed null sequences which are either {\em well behaved}  and {\em relevant} null sequences ${\cal S}^a_{0}$ and those which are {\em irrelevant}  null sequences relative to the chosen scale $a$. The irrelevant null sequences  define the renormalized null set ${\cal S}^0_{0}$, when the well behaved sequences in ${\cal S}^a_{0}$ now provide new asymptotic structure that we are concerned here. The non-null Cauchy sequences $\{b_n^\pm\}$ is subsequently interpreted as renormalized effective values awarded to null sequences and shown to have the property of an ultrmetric absolute value \cite{dss}. As a convention, we assume the product of two or more null sequences to belong to the renormalized null set ${\cal S}^0_{0}$, i.e., $a_1a_2\in {\cal S}^0_{0}$,  for $ a_1,a_2\in {\cal S}_0 $, etc. However, a multiplicative structure in the relevant null sequences ${\cal S}^a_{0}$ is induced via the discrete ultrametric valuation, via proposition (\ref{ultram}).
}
\end{remark}

\begin{example}
Fix $a:=n^{-1}$ as the privileged scale. Let ${b}^\pm=n^{-(1+kn^{\mp\alpha}+o(n^{-1}))}, \ k>0, \ 0<\alpha<1$. Then ${b}^\pm$ are  good sequences since $v_n({b}^\pm)= kn^{\mp\alpha}+o(n^{-1})$. The sequence of the form $c=n^{-(1+\sin\pi n+o(n^{-1})}$ is an example of irrelevant sequence since $v_n(c)=\sin\pi n +o(n^{-1})$ is oscillating. Sequences of the form $d=n^{-(1+kn^{-\alpha}+o(n^{-1}))}, \ k>0, \ \alpha\geq 1$ also constitute a class of redressed null sequences since $v_n(d)=o(n^{-\alpha}), \ \alpha\geq 1$. By convention, sequences of the form $n^{-2}, \ n^{-3}$ etc are irrelevant. By the same token, sequences ${b^\pm}^2$ are also irrelevant.
\end{example}

As a consequence, any null sequence ${\bar a}\in {\cal S}_0$ gets a {\em renormalized effective value}, called the {\em asymptotic renormalized valuation}, and denoted by the sequence $v_n(\bar a)={ |\log_{a_n^{-1}}|{\bar a}_n/a_n| \ |}, \ n\geq N$. Relative to the chosen scale $a$ and the associated renormalized valuations $v_n(\bar a), \ \forall {\bar a}\in {\cal S}_0$, the set of {\em relevant null sequences} $\tilde{\cal S}^a_0$ now has the following refined partition.

\begin{proposition} {\bf Refined Partition of Null Sequences} {\rm \cite{dss}}\label{rfn}
 The choice of a privileged scale (sequence) $a=\{a_{n}\}\in { {\cal  S}}_{0}$
introduces a polarizing effect and partitions ${ \tilde{\cal  S}}^a_{0}$ into equivalence
classes having the renormalized representations given by

${\tilde{\cal S}}^a_{0}=\{ \{A_n^0\},\{A_{n}^{\infty
}\},\{A_{n}^{-}\}, \{A_{n}^{+}\}, \{A_{n}^{^{\prime }}\}\}$ where

$(1.a)  \ \ |A_{n}^{0 }|= a_{n}\times a_{n}^{{\alpha_n }(1+o(1))}(1+o(1))$

$(1.b) \ \ |A_{n}^{\infty }|= a_{n}\times a_{n}^{{\beta_n }(1+o(1))}(1+o(1))$

$(1.c) \ \ |A_{n}^{+}|= a_{n}\times a_{n}^{-k^n_{+}(1+o(1))}(1+o(1))$

$(1.d) \ \ |A _{n}^{-}|= a_{n}\times a_{n}^{k^n_{-}(1+o(1))}(1+o(1))$

$(1.e) \ \ |A_{n}^{^{\prime }}|= a_n\times a_{n}^{\pm \gamma^{\pm}_n(1+o(1))}(1+o(1))$

\noindent where $\ n\geq N, $  for sufficiently large $N $. Moreover, $|\alpha_n|\rightarrow | 0$, $|\beta_n|\rightarrow \infty$ or does not exist, $|k^n_-|\rightarrow |k_-|>1$,  $|k^n_+|\rightarrow |k_+|: \ 0< k_+<1$, and $\gamma^{\pm}_n\rightarrow \gamma:, \ 0\leq \gamma\leq 1$, for $n\rightarrow \infty$.  Further, we set $S_0^0=\{\{A_n^0\}, \{A_n^{\infty}\}\}$, that also include the product of null sequences, by convention.
\end{proposition} 

\begin{remark}\label{not}
{\rm The 'small oh' symbol $o(1)$, in the above lemma and also in the following, and the 'big oh' $O(1)$, has the standard text book meaning: $o(1)$ denotes a null sequence, when $O(1)$ indicates the order of bound. Further, the renormalized representations of null sequences, $A_n^\pm$, for instance, are defined up to a slowly varying multiplicative factor $M _n$ such that $\frac{\log M_n}{\log a_n}=o(1)$, that is of course taken care of by the $o(1)$ factors in the exponents of the r.h.s.}
\end{remark}

Every `nontrivial' null sequence $\bar a \in {\cal S}_0$ gets mapped, by proposition (\ref{rfn}), into ${\tilde{\cal S}^a_0}$ either as $A_n^+$ or $A_n^-$. For definiteness, let it be $A_n^+$. Then, in virtue of the {\em natural} multiplicative structure that exists in the  renormalized null set ${ \tilde{\cal  S}}^a_{0}$, one also assigns the original sequence $\bar a\in {\cal S}_0$ a {\em dual} image $A_n^-$ by the relations 

\begin{equation}\label{ds1}
A_n^-/a_n =\lambda_n a_n/A_n^+, \ {\rm and} \ k_n^-\cdot k_n^+=\mu_n, \ n\geq N
\end{equation}
where $\lambda_n\sim o(1)$ (or slowly varying constant $\sim O((|\log a_n|)^k), \ k$ real constant 
and $\mu_n\sim o(a_n^{\alpha}), \ 0<\alpha<1$ are  any two positive null sequences,  ($\mu_n$ is slowly varying,   vanishing at a slower rate than $a_n$). Each renormalized sequence, for instance, $A^+$ is thus associated, in principle, with a two parameter family of {\em dual} renormalized sequences denoted explicity as $A^-(\lambda, \mu)$. In a specific problem or application, one, however, expects to weaken this huge arbitrariness to countably many or even to unique choice (upto slowly varying constants, remark (\ref{not})) for the dual elements (see examples below). For the sake of notational ease, we shall omit such explicit notations in favour of the shorter notation $A^-$ with the assumed implicit dependence on  the indicated parameters.

\begin{definition} {\bf Duality Structure}
The renormalized null set  ${\tilde{\cal S}}^a_{0}$ equipped with the renormalized valuations $v^\pm_n(\bar a)$ and the associated refined partition, as defined by proposition (\ref{rfn}), together with the natural multiplicative structure  (\ref{ds1}) is called the ``asymptotic duality structure" induced by the chosen scale. 
Relative to this  duality structure,   the sequences $A_n^0, \ A_n^{\infty}$ and $A_n^+, \ A_n^-$ are respectively dual to each other, when $A_n^{\prime}$ are self dual.
\end{definition}

\begin{example}\label{exdl}
{\rm 
(i) Let $a_n=n^{-1}$ be the scale and $A^+=n^{-1+k_n}$ where $k_n$ be a sequence converging in (0,1]. Then $v_n(A^+)=k_n\rightarrow k,$ as $n\rightarrow \infty$ and $k\in (0,1]$. Then a class of self dual sequences are $A^-=\lambda_n\cdot n^{-1-\bar k_n}$ with $v_n(A^-)=\bar k_n\rightarrow k$, when $\frac{\log \lambda_n}{\log n^{-1}}\rightarrow 0$ as $n\rightarrow \infty$. 

(ii) For a non self-dual pair, let $a_n=3^{-n}$ be the scale and $A^-_i=3^{-n}\cdot 3^{-n(i_n^{-2}2^{m_n})}$ so that $v_n^-=i_n^{-2}2^{m_n}$, where $i_n=0,1,\ldots,  2^{m_n}-1$, $i_n, \ m_n$ are slowly varying integers relative to $n$. Let $\lambda_n=3^{-n\kappa_{m_n}}$ and $\mu_{m_n}=i_n^{-1}$. Then the dual sequence $A^+_i$ is given by $A^+_i=3^{-n}\cdot 3^{-n\kappa_{m_n}}  \cdot 3^{n(i_n 2^{-m_n})}$, provided $\kappa_{m_n}=i_n^{-2}2^{m_n}$. Notice that (a)  $\lambda_n$ is fast decaying, and $\mu_n$ slowly varying, compared to the scale $a_n$,  (b) $v_n(A^-_i)\approx \kappa_{m_n}$ and $v_n(A^+_i)\approx \mu_{m_n}/\kappa_{m_n}=i_n 2^{-m_n}$, so that the duality structure (\ref{ds1}) is (approximately) respected, when $v_n^+<<1$ is neglected compared to $v_n^->>1$ (see also theorem (\ref{charac}), proposition (\ref{dual}) and remark (\ref{dual1}), for more details).}
\end{example}

\subsection{ Refined Order relation} 

The usual order relation $\leq$ in $\bf R$ gets extended over the renormalized null set via the asymptotic renormalized valuation $v_n$: two renormalized sequences $A$ and $B$ are linearly ordered viz., $A\leq B$ if and only if $v_n(A)\leq v_n(B), \ n\geq N$. Further, two renormalized sequences $A, \ B$ are equivalent i.e. $A\sim B$ if $A-B\in {\cal S}_0^0$. The archimedean field $\bf R$ now is extended to a non-archimedean field ${\bf ^*R}_a$ that is realized as the quotient space ${\bf ^*R}_a\equiv {\cal S}\backslash {\cal S}_0^0$ where $\cal S$ is the ring of Cauchy sequences of real numbers \cite{katok}. Equivalently, ${\bf ^*R}_a=[{\cal S}]$, the set of equivalence classes, when two sequences in $\cal S$ are said to be equivalent if they differ by a  sequence of ${\cal S}_0^0$. Clearly, ${\bf R}\subset {\bf ^*R}_a$, and the nonarchimedean field extension ${\bf ^*R}_a$ enjoys new asymptotic elements equipped with nontrivial asymptotic valuations. As a consequence,  the  field $^*\bf R$ is, by construction,  linearly ordered and Cauchy complete non-archimedean extension of $\bf R$. For more details, we refer to \cite{dss}.


\subsection{Duality in Extended Field} 

Since $\bf R$ is a subfield of ${\bf ^*R}_a$, the nontrivial asymptotic elements of ${\bf ^*R}_a$ are accessed via the usual limit process $x\rightarrow 0, \ x\in \bf R$ as detailed below.  

Let $x\in {\bf R}$ and $x\rightarrow 0$. Then there exists a scale $\delta=\delta (x)>0$ such that $0<\delta(x)\leq x$. One now considers the extended field  $  {\bf ^*R}_\delta \supset {\bf R}$, so that the null set $[-\delta,\delta]$ gets renormalized into a larger interval ${\bf ^*O}:= [-^*\delta, ^*\delta] \supset [-\delta,\delta]$. As a consequence, an  asymptotic real variable $x$ is mapped into a pair of   renormalized components $x\mapsto X^\pm_\delta$ satisfying $0<X^-_{\delta}(x)<\delta(x)<x \leq X^+_{\delta}(x)$,  and having nontrivial renormalized asymptotic valuations,  given by 
\begin{equation}\label{vnorm}
v_{\delta}^\pm(X^\pm_\delta(x))=\bigg  |\log_{\delta}|\frac{X^\pm_{\delta}}{\delta}|\bigg |(1+o(1)) 
\end{equation}
This pointwise asymptotic extensions must also respect the  {\em duality structure} defined by 
\begin{equation}\label{ds2}
  X^-_{\delta}X^+_{\delta}=\lambda_{\delta}\delta^2, \ {\rm and} \ 
	v_{\delta}^-(X^-_\delta)v_{\delta}^+(X^+_\delta)=\mu_{\delta}.
\end{equation}
where $\lambda_{\delta}\sim o(1)$ is a null sequence or a slowly varying constant $\lambda_{\delta}\sim O((|\log\delta|)^k), \ k$ 
a real exponent, and  $\mu_{\delta}\sim o(\delta^{a}), \ 0<a<1$ is  slowly varying   depending on $\delta$. 

Notice that, even as $x\in \bf R$ and the associated dual pairs $X^\pm_\delta\in {\bf ^*R}_\delta$ are all asymptotic, i.e. vanishing in the limit $\delta(x)\rightarrow 0$, the valuations $v_{\delta}^\pm(X^\pm_\delta(x))$ are nevertheless  finite, or slowly varying (vanishing at slower rates)  real variables. As a consequence, any vanishing or divergent quantity would enjoy nontrivial finite asymptotic values that should have significant applications in nonlinear, asymptotic dynamics. Further, the asymptotic pairs $X^\pm_\delta$  also admit the natural renormalization representations
\begin{equation}\label{renom} 
X^\pm_\delta(x)=\delta(x)\cdot\delta(x)^{\mp v_{\delta}^\pm(X^\pm_\delta)(x)}(1+o(1))
\end{equation}

To emphasize, {\em the  variables, $x$ and  $\delta$ are arbitrarily small, asymptotic quantities, but are thought to be held fixed, at the level indicated by the scale  $\delta$. The above field extension ${\bf ^*R}_\delta$ and associated renormalized asymptotics are realized and valid at the fixed scale $\delta$}. A limit of the form $x\rightarrow 0$ and/or $\delta\rightarrow 0$ then, in fact, induces a concomitant variation in the field extension ${\bf ^*R}_\delta$ itself. In the following, most of our discussions  assume the presence of a fixed scale $\delta$. Explicit scale variation is, however, used in proposition (\ref{dual}) and in the RG interpretation in Sec. 2.3.

As a consequence, and for simplicity, but with a slight abuse, of notation, we shall hence forth denote by $0<\tilde x<\delta<x$ the dual pair $X^\pm_\delta$ respectively  in ${\bf ^*R}$ ( suffix $\delta$ indicating explicit scale dependence is omitted). As should be clear, the above nontrivial dual pair, when restricted in $\bf R$, however, collapses (projects) into the original limiting real variable $x\in \bf R$. Moreover, the nontrivial asymptotic valuations are essentially functions of the rescaled logarithmic variable $\xi=\log x/\delta$, so that we set $v(x):=v_{\delta}(X^+)=v(\xi)$ and $\tilde v(\tilde x):=\tilde v_{\delta}(X^-)=\tilde v(\xi)$.

\begin{proposition}[\bf Invariance Properties]{\rm \cite{dss}}\label{invprop}
\label{invariance}
Let $x$ be an asymptotic variable in ${\bf ^*O}\subset {\bf ^*R}$ and $v(x):=v_{\delta}(x)$ be the asymptotic valuation of $x\in {\bf ^*R}$ relative to the scale $\delta$. Then \\
(1). $v(kx)=v(x)$, for any constant $k$, \\
(2).  $v_{k\delta}(x)=v_{\delta}(x),$ for any constant $k>0$, \\
(3). $v(x^{-1})=v(x)$, \\
(4). $v(x+x_0)=v(x)$, for any $x_0$ such that $0<\delta\leq |x_0|\leq |x|$, 

which are all valid up to an additive null sequence $o(1)$.
\end{proposition}

The proof of above properties follow directly from the definition (\ref{vnorm}). In the case of inversion invariance (3), the scale for $x^{-1}$ is ${\delta}^{-1}$. The property (1) is the scale invariance and (2) is the invariance under scale reparametrization. Because of (1) and (2) of the above proposition, the representations  (\ref{vnorm}) are called the {\em normal} forms of the dual asymptotic variables in ${\bf ^*R}$. The translation invariance (4) is responsible for the ultrametric property of the valuation.

We recall that, by convention, the product of two null sequences $a,b\in {\cal S}_0$ is assumed to be irrelevant and hence is null even in the present setting. A multiplicative structure in $\bf ^*O$ is, however, induced by the renormalized valuation $v$. 
\begin{definition}\label{prod}
The product of two asymptotics $x_1$ and $x_2$ in $\bf ^*O$ is defined by the renormalized asymptotic $x_1ox_2:=\delta\cdot\delta^{-v(x_1)v(x_2)}$, so that $v(x_1ox_2)=v(x_1)v(x_2)$. Such a product is called renormalized product for null asymptotics.
\end{definition}

\begin{lemma}
The product $x_1ox_2$ belongs to $\bf ^*O$, and hence exists uniquely and well-defined.
\end{lemma}
The proof is immediate from the definition. The algebraic properties of this definition of multiplication will be studied elsewhere. 

\begin{proposition}[\bf Ultrametric Property]{\rm \cite{dss}}\label{ultram}
The asymptotic valuation $v:{\bf ^*O}\mapsto \bf R^+$ acting on  ${\bf ^*O}:=(-^*\delta,^*\delta)\subset {\bf ^*R},$ and $(-^*\delta,^*\delta)\supset (-\delta,\delta), \ \delta\rightarrow 0$,   is a discretely valued  ultrametric norm (absolute value). Furthermore,  ${\cal O}=(-v(\xi),v(\xi))\delta\log \delta^{-1} +(-1,1)o(\delta)$ is the associated asymptotic extension of the linear neighbourhood $(-\delta,\delta)\subset \bf R$ of 0, when $v$ leaves in the right neighbourhood of $\delta$ in ${\bf ^*R}$.
\end{proposition}

\begin{definition}\label{defrmd}
The asymptotic extension of $0=\{[-\delta,\delta], \ \delta\rightarrow 0\}\in \bf R$ into ${\bf 0}=0+\cal O$ is called the asymptotic deformation of the point 0. This definition extends over each point $r\in \bf R$, by translation. The associated asymptotic deformation of $\bf R$ induced by the asymptotic valuation $v$ is denoted by $\cal R$.
\end{definition}

\begin{remark}{\rm The valuation can therefore be either trivial or can have a discrete, non-dense or dense,  set of distinct real values, that can spread continuously over the extended neighbourhood of 0. In the later case, the valuation $v$ has the structure of a devil's staircase function (for a proof see \cite{dss}, for a new characterization proposition (\ref{dual})). Further, the nontrivial extensions of 0, i.e. $\bf ^*O$ in $\bf ^*R$ and ${\cal O}\subset \cal R$ are quite distinct; $\bf ^*O$ is totally disconnected in the ultrametric topology (c.f. proposition (\ref{comp})), when $\cal O$ is {\em connected}, and generally, denotes a fractal curve. In the special case, when $v$ is the {\em trivial ultrametric}, the said fractal curve {\em degenerates into a connected line segment}. For completeness, we include the proof of the above proposition.
}
\end{remark}

\begin{proof}
First, we prove strong triangle inequality \cite{katok} for  asymptotics in $(0,\delta)$. Let $0<\tilde x_1<\tilde x_2<\tilde x_1+\tilde x_2<\delta$ such that $v(\tilde x_1)>v(\tilde x_2)$. We have, $v(\tilde x_1+\tilde x_2)=\bigg |\log_{\delta^{-1}} \frac{\delta}{\tilde x_1+\tilde x_2}\bigg |=|\log_{\delta^{-1}} \frac{\tilde x_1(1+{\tilde x_2}/\tilde x_1)}{\delta}|=v(\tilde x_1)$, in the limit $\delta\rightarrow 0$, since the second term $\frac{\log(1+{\tilde x_2}/\tilde x_1)}{\log \delta^{-1}}$ vanishes in that limit, the numerator $\log(1+{\tilde x_2}/\tilde x_1)$ being at most finite. As a result, $v(\tilde x_1+\tilde x_2)\leq{\rm max}\{v(\tilde x_1),v(\tilde x_2)\}$ \cite{katok}. 

Next, in virtue of the strong triangle inequality, one can assign discretely valued exponential valuation of the form $v(\tilde x_1)=\delta^{a(\tilde x_i)}$ for a $0<\delta<1$, satisfying  $a(\tilde x_1+\tilde x_2)\geq{\rm min}\{a(\tilde x_1),a(\tilde x_2)\}$ and $a(\tilde x_1\tilde x_2)=a(\tilde x_1)+a(\tilde x_2)$, the value group being isomorphic to the additive group of  integers \cite{katok}.  
The relation $v(\tilde x_1o\tilde x_2)=v(\tilde x_1)v(\tilde x_2)$ then follows naturally, when the product of two asymptotics is interpreted as the renormalized product, definition (\ref{prod}). 

Similar arguments also apply  for asymptotics satisfying $\delta\leq x_1<x_2<x_1+x_2$ and $v(x_2)>v(x_1)$.

Finally, in the limit $x, \ \delta\rightarrow 0$ ($0<\delta<x$), $\delta\in {\bf R} \ \mapsto \ {\bf ^*R}\ni\delta^*=\delta\cdot\delta^{- v(\xi)}$, satisfying the inequalities $0<\delta<x\leq \delta^*$, so that nontrivial asymptotic limit of $\cal O$ has the form indicated as above.
\end{proof}

As a consequence, the open balls in $B(a, r)=\{x\in {\bf ^*O}: \ v(x-a)<r\}$ are clopen i.e. both closed and open and hence, $\bf ^*O$ can at most be a countable union of  clopen balls and hence totally disconnected \cite{katok}. Further, $\bf ^*O$ is also bounded under the valuation $v$, since $0<v(x)<1$, by definition, for an asymptotic $x$ satisfying $0<\delta<x\leq \delta^*$. Asymptotics $\tilde x: \ v(\tilde x)>1$ must satisfy $0<\tilde x<\delta$ and hence are identified essentially as null \cite{dss}. Accordingly, it follows that 

\begin{proposition}\label{comp}
The extended set $\bf ^*O$ is totally disconnected, compact and complete metric space in the ultrametric valuation.
\end{proposition}
\begin{proof}
Completeness follows from the fact that every Cauchy sequence is eventually constant, in the discretely valued ultrametric, and hence convergent \cite{katok}.
\end{proof}

\begin{corollary}
The nonarchimedean field $\bf ^*R$ is a complete and locally compact metric space under the asymptotically extended metric defined by the usual norm $|\cdot|$ acting in $\bf R$ and the asymptotic valuation $v$ acting on the asymptotic neighbourhood $\bf ^*O$. 
\end{corollary}

Based on this observation and also because of the duality structure (\ref{ds2}), we have following classifications of the renormalized asymptotics and the associated valuations.

\begin{definition}[{\bf Classification}]\label{class}
Renormalized asymptotics $\tilde x$ and $x$ in $\bf ^*O$ are respectively (i) self dual,  (ii) weakly self  dual, (iii) critically self dual, or (iv) strictly dual to each other if (i) $\tilde v(\tilde x)= v(x)$,	(ii) $\tilde v(\tilde x)=\alpha(\delta) v(x)$, (iii)  self dual  limit $\alpha\rightarrow 1$ of (ii), if it exists, or (iv) $\tilde v(\tilde x)v(x)=\mu(\delta)$, where $ \alpha> 0, \ \mu>0$ are slowly varying parameters depending on $\delta$. The self (or strictly) dual  pair $\tilde x  {\rm \ and}\ x$ are called self (or strictly) dual conjugates. These definitions are valid and meaningful up to a remainder term $o(\delta)$.
\end{definition}

\begin{remark}{\rm
In an application, we say an asymptotic $x$ is dominantly self dual if the magnitude of $v(x)$ is much greater than $O(\delta)$. In such a case, weakly self dual or strictly dual asymptotics may contribute to a nontrivial remainder, intermediate to the dominant self dual term and $O(\delta)$. When, on the other hand, a self dual satisfying $v(x)\leq O(\delta)$, and hence is insignificant, strictly dual asymptotics would play dominant role in determining the nontrivial neighbourhood structure of $\cal  O$.}
\end{remark}

\begin{example}{\rm 
The self dual pair of Example 1(i) is actually immersed into a family of weakly self dual elements given by $A^{\alpha -}=\lambda_n\cdot n^{-1-\alpha_n\bar k_n}$ with $v_n(A^{\alpha -})=\alpha_n\bar k_n\rightarrow \alpha k +o(n^{-1})$ so that $v(A^{\alpha -})=\alpha v(A^+) +o(n^{-1})$. The pair is critically  self dual if the critical limit $\alpha\rightarrow 1$ exists.
}
\end{example}

\begin{theorem}[{\bf Characterization}]\label{charac}
Renormalized asymptotics $\tilde x$ and $x$ in $\bf ^*O$ are respectively (i) weakly self dual, (ii) self dual or (iii) strictly dual to each other, if and only if the scaling parameter $\lambda_{\delta}=\lambda(\delta)$ in (\ref{ds2}) has the form  \\
  {\rm (i)} $\lambda(\delta) = \lambda_0 (|\log \delta|)^{k}$ (1+o(1))\\
  {\rm (ii)} $\lambda(\delta) = \lambda_0 $ (1+o(1))\\
	{\rm (iii)} $\lambda(\delta) = \lambda_0 {  \delta}^{\kappa}$(1+o(1)) \\
where $\lambda_0=\lambda_0(\delta)>0, \ k=k(\delta) \ {\rm real \ and}, \ \kappa=\kappa(\delta)>0$ are  slowly varying constants, up to  additive null $o(1)$  sequences . Consequently, the weakly self dual  asymptotic valuation $v(x)$ has the universal form given  by 
\begin{equation}\label{wsd}
v(x)=\gamma(\delta) \bigg|\frac{\log\log\delta^{-1}}{\log \delta}\bigg |(1+o(1)), \ \gamma(\delta)=\frac{k(\delta)}{(\alpha(\delta)-1)}>0, 
\end{equation}
so that $k\lessgtr 0 \ \Leftrightarrow \ \alpha\lessgtr 1$.	
The valuation for  critically self dual asymptotic is then given by (\ref{wsd}), but with the  limiting form of the  multiplicative factor $\gamma(\delta)\rightarrow \tilde\gamma(\delta)=\lim\frac{k(\delta)}{(\alpha(\delta)-1)}$, when the limit exists, 
where the limit in the right hand side corresponds to the {\em critical self dual limit}: $\alpha\rightarrow 1, \ k\rightarrow 0$, satisfying the above finiteness constraint. For a dual asymptotic, the valuation, on the other hand,   is given by the quadratic equation
\begin{equation}\label{dl}
v(x)^2+\kappa(\delta)v(x)-\mu(\delta)=o(1) 
\end{equation}
Finally, for a self dual asymptotic, (ii) yields 
\begin{equation}\label{sde}
v(x)=\tilde v(\tilde x) + o(1)
\end{equation}	
as expected.
\end{theorem}

\begin{proof}
First equality of the duality structure (\ref{ds2}), together with (\ref{vnorm}), implies that $v(x)=\tilde v(\tilde x)+ \frac{\log\lambda}{\log\delta^{-1}} +{ lower \ order \ terms}$. The cases (i) and (iii) then uniquely identify the weakly self dual and strictly dual asymptotics following respectively (\ref{wsd}) and (\ref{dl}). The critically self dual case  then follow from (\ref{wsd}) in the {\em critical limit} defined by $k\rightarrow 0$ and $ \alpha\rightarrow 1$ simultaneously, so that $\tilde\gamma(\delta)=\underset{\alpha\rightarrow 1,  k\rightarrow 0}\lim \gamma(\delta)$ exists finitely. Notice also that the limit $\kappa\rightarrow 0$ in the duality equation (\ref{dl}) is consistent with this critical self duality, when $\sqrt\mu=\tilde\gamma(\delta) |\frac{\log\log\delta^{-1}}{\log \delta}|$. Finally, the case (ii) trivially yields the self dual equation (\ref{sde}) from definition, when $o(1)\equiv o((\log\delta^{-1})^{-1})$.
\end{proof}

Two interesting consequences of the above classifications are formulated in the following corollaries.

\begin{corollary}
Given a nontrivial asymptotic $x=\delta\cdot\delta^{-v(x)}$, there exists respectively weakly self dual $x_w$ and strictly dual $x_d$ asymptotics given by $x_w=\delta\cdot\delta^{\alpha(\delta)v(x)}$ and $x_d=\delta\cdot\delta^{\mu(\delta)/v(x)}$. The parameters $\alpha$ and $\mu$ respect the constraints of the above theorem.
\end{corollary}

\begin{corollary}\label{wsdn}
Weakly self dual asymtotics ${\tilde x}_{w}$ and ${x}_{w}$ must also respect, for consistency, the duality condition (iv) of definition (\ref{class}) so that $v({x}_{w})=\sqrt{\mu/\alpha}$ and $\tilde v({\tilde x}_{w})=\sqrt{\alpha\mu}, \alpha>0$. Further, apart from the  critical self dual limit $\alpha\rightarrow 1$, there exists another class of nontrivial, strictly dual, arithmetical fixed points $\alpha\rightarrow \phi^2$,  where $\phi$ is a  quadratic irrational satisfying $\phi^2 -r\phi -1=o(1)$ for a rational $r=p/q, \ q\neq 0$ where $(p,q)=1$ coprime, and the discriminant $p^2+4q^2$ is not a perfect square. However, $\phi$ can be a more general irrational  when  $r$ is irrational. For a perfect square discriminant,  nontrivial rational fixed points correspond to  rational values of $\phi$.
\end{corollary}
\begin{definition}
 Strictly dual  arithmetical fixed points of Corollary \ref{wsdn} are called weakly dual arithmetical asymptotics. The  rational fixed points, in particular, lead to a class of weakly dual arithmatical asymptotics, distinct from  those close to the critical self dual limit and are known as weakly dual rational asymptotics.
\end{definition}
\begin{proof}
The first statement follows directly from definitions of weakly self dual and strict duality conditions of definition (\ref{class}). The second property follows from the {\em unique} scaling invariance of the quadratic duality equation (\ref{dl}), that should hold for  weak (self) duality: $\kappa(\delta)=r\sqrt{\mu(\delta)}, \ r=p/q$,  so that $\alpha\rightarrow \phi^2$ and $\phi^2 - r\phi -1=o(1)$. The positive root is a quadratic irrational if and only if $p^2+4q^2\neq m^2,$ for an integer $m$. The critical self dual limit corresponds to $r=0$.
\end{proof}

\begin{example}
Here, we give an example of a pair of weakly  dual conjugates  where  self dual limit $\alpha\rightarrow 1$  fails: $x_w=a_1\delta\cdot(|\log\delta|)^{-|\frac{k}{\alpha-1}|}$ and ${\tilde x}_w=a_2\delta\cdot(|\log\delta|)^{|\frac{k\alpha}{\alpha -1}|}$ for positive constants $a_1,\ a_2, \ |k|\neq 0$, so that $v(x_w)=|\frac{k}{\alpha-1}||\frac{\log\log\delta^{-1}}{\log\delta}|$ and $v(\tilde x_w)=|\frac{k\alpha}{\alpha-1}||\frac{\log\log\delta^{-1}}{\log\delta}|$, for $\alpha>0 $. 

For an explicit example of a weakly  dual rational asymptotic, consider asymptotics with valuations $v(x_w)=\sqrt\mu/2$ and $\tilde v(\tilde x_w)=2\sqrt\mu$, then $r=3/2$ and hence $\phi=2 \ {\rm or} \ -1/2$ and $\alpha=4 \ {or} \ 1/4$. In fact, for each rational $\phi_r$ and its negative inverse $-1/\phi_r$,  there exists a special value of the rational $r$, so that the pair ($\phi_r$, $-1/\phi_r$)  constitutes two  rational solutions of the equation $\phi^2-r\phi-1=0$. Such weakly  dual asymptotics, away from self dual limit and also distinct from weakly self dual asymptotics, are generally intermediate to the critical self dual  and weakly dual asymptotics, and will be studied separately. 
\end{example}

\begin{remark}\label{self}
{\rm 
1. The family of self dual asymptotics is ultrametrically trivial, and have unique self dual valuation $v_s(x)$ for any self dual $x_s$.

2. The unfolding of self dual to (nontrivial) weakly  dual, rational asymptotics and then finally to more general weakly dual (irrational) asymptotics is an interesting phenomenon. More in depth analysis in the context of dynamical systems and number theory would be considered separately.

3. The weakly  self dual asymptotics can further be reclassified according to slower logarithmic variations of $\lambda(\delta)$ of the form
\begin{equation}
\lambda(\delta) = \lambda_0 (\log^n \delta)^{-k} (1+o(1))
\end{equation}
where $\log^n (\cdot)= \log\log\ldots { n \ times} \ (\cdot)$, and respecting analogous weakly self dual and arithmatically dual fixed points respectively. 
}
\end{remark}

Next, for a  self dual (critically or not) asymptotic $x\in {\bf ^*O}$, the valuation $v(\tilde x)=v(x)$ is a  constant, independent  of any $\tilde x\in (-\delta,\delta)$, but is, in fact, a scale invariant smooth function of the scaling parameter $\xi=\log x/\delta<1, \ \delta=\delta(x)$, satisfying the invariance properties of the propositions \ref{invariance} and \ref{ultram}.

\begin{proposition}[{\bf Self dual $\Rightarrow$ Smoothness}] \label{sd}
The discrete ultrametric property, proposition \ref{ultram}, of the valuation implies that the self dual valuation maps the extended null set ${\bf ^*O}\subset {\bf ^*R}$ into ${\cal O}\subset \cal R$  as a connected segment of a smooth curve equipped with the smooth variable $v(\xi)\propto \xi^{i}, \ 0<\xi=\log x/\delta<<1$, $i$ being a non-negative integer.
\end{proposition}
\begin{proof}
Self dual valuation $v(\xi), \ \xi=\log x/\delta<1$ is a trivial ultrametric, being independent of $\tilde x\in (-\delta,\delta)\subset (\delta^*,\delta^*)$, and hence $\bf ^*O$ is totally disconnected. On the other hand, for a fixed asymptotic $x\in \bf R$ and the corresponding scale $\delta$,  the valuation $v$ maps the extended set generated by the renormalized self dual asymptotics $X^\pm=\delta\cdot \delta^{\pm v(\xi)}\in {\bf ^*O}$ into a connected {\em smooth} curve segment in the interval defined by $(-v(\xi),v(\xi))\subset \cal R$ (by proposition (\ref{ultram}) and Remark 3), equipped with the unique renormalized  variable $v(\xi)$ (c.f. remark \ref{self}(1)) that has, at the most, the default form  $v(\xi)\propto \xi^i, \ i$ being a nonnegative integer, and so extends {\em smoothly} to $v(0)=0$. Notice that $v(\delta)=0$ and $v(\delta^*)=v(\xi)$, since $\delta$ and a general $X^+$ can relate only linearly, i.e. $ \delta^*=kX^+, \ k>0$, a constant parameter (by proposition \ref{invprop}). Further, the variability of $v$ follows from that of $x$ and $\delta$ by $\xi=\log x/\delta$.
\end{proof}

\begin{corollary}\label{pws}
Weakly self-dual asymptotics $x$ and $\tilde x$, around a critical  self dual asymptotic, can proliferate to a set of, at most, a finitely many nontrivial rational fixed points,  and induce a pair of piece-wise smooth scaling variables $v(\xi)$ and $\tilde v(\xi)$ to parametrize the connected set $\cal O$. Further, $\tilde v(\xi)=\alpha(\delta)v(\xi)$.
\end{corollary}

\begin{proof}
Let $x$ and $\tilde x$ be a self dual pair so that $v(x)=v(\tilde x)\equiv v(\xi)$. Denote the associated {\em rational }, weakly  dual asymptotics by $x_i$ and $\tilde x_i$, so that $\tilde v(\tilde x_i)= \alpha_i(\delta)v(x_i)$ together with constraints $v(x_i)=\sqrt\mu_i/\phi_i$ and $\tilde v(\tilde x_i)=\phi_i\sqrt\mu_i$, where $\phi_i$ is a rational fixed point  satisfying $\phi^2-r_i\phi-1=0, \ r_i=p_i/q_i, \ q_i\neq 0$ (by corollary (\ref{wsdn})). Further, $p_i^2+4q_i^2=m_i^2$, for an integer $m_i$. Set of such rational fixed points {\em can not be dense}, since otherwise, by the completeness of $\bf ^*O$,  the value set of $v$ would contain irrational fixed points, corresponding to weakly dual irrational asymptotics (to be considered separately in proposition (\ref{dual})).

 Introduce notations $v(x_i):=v_i(\xi)$ and $v(\tilde x_i):=\tilde v_i(\xi)=\alpha_i(\delta)v_i(\xi)$. The non-dense value set of $v$   admits now a finitely many  nontrivial  limit points, depending on $r_i$, other than the one corresponding to the self dual limit $\alpha(\delta)\rightarrow 1$, and $\ r_i=0$. Because of ultrametric property, (viz,  (i) $\bf {^*O}$ is  a countable union of disjoint clopen balls and hence disconnected, and (ii) $v$ continuous on $\bf {^*O}$), the map $v:{\bf ^*O}\mapsto \cal O$ could, at most,  be  composed of piece-wise continuous (and may very well be affine) pieces, $v_i(\xi)$, that is smooth on the $i$th component ball. 
\end{proof}

\begin{example}[\bf Self dual] {\rm 
i) The dominant asymptotic of the prime counting function, $\Pi(n)\sim \frac{n}{\log n}(1+o(1))$, given by the prime number theorem, constitutes an example of (critical) self dual asymptotic with self dual  valuation $v(\Pi)=\frac{\log\log n}{\log n} +o(1)$. The $n$th prime $P(n)\sim n\log n(1+o(1))$ constitutes the self dual conjugate with identical valuation. The possible role of weakly dual arithmetical asymptotics in PNT will be investigated separately.

ii) Self dual asymptotics should also arise naturally in nonlinear differential equations with nontrivial limit sets. For instance, the stable limit cycle of the Van der Pol oscillator is an asymptotic state that is reached only in late time $t\rightarrow \infty$. The self dual valuation $v(\tau)\propto \tau, \ \tau=\epsilon t\sim O(1)$, ($\epsilon<1$ being the nonlinearity parameter), then offers a natural scaling variable to parametrize the  asymptotic cycle trajectory \cite{dpp}, that is not generally available in the standard perturbative analysis of the problem .
Analogous applications of weakly self dual valuations is expected to have significance in period doubling bifurcation in an externally driven oscillator.}
\end{example}

\begin{remark}\label{hm}
{\rm 
Renormalization mapping $x\mapsto v(\xi)$ for $x\rightarrow 0$ extends the vanishing linear neighbourhood $\{(-\delta,\delta)|\delta\rightarrow 0\}$ into a nontrivial, finite interval $(0,v(\xi))$ (or $(-v(\xi),v(\xi))$) in which points are parametrized by the natural renormalized variable $v(\xi)$. For self dual asymptotics, $v(\xi)$ is a smooth variable, for weakly  dual case it is piece-wise smooth, but for strictly dual asymptotics, $v(\xi)$ would in general be nonsmooth (or singular) and parametrizes a fractal curve in the asymptotic neighbourhood of 0 (say). As it would transpire, the function $v(\xi)$ would have the structure of a Lebesgue-Cantor staircase type function $S^s_F$ that is defined from the Hausdorff measure or, equivalently, from the integral mass function, defined in \cite{parv1,parv2},
\begin{equation}\label{isf}
S_F^s(\xi)=\begin{cases}
\gamma^s(F,p_0,\xi), & \xi\geq p_0 \\
-\gamma^s(F,\xi,p_0), & \xi < p_0
\end{cases}
\end{equation}
corresponding to the integral mass function, either of zero dimensional fractal (Cantor) set or of a parametrized fractal curve $F$: $\gamma^s(F,a,b)=\underset{\delta\rightarrow 0}\lim{\ \inf \sum_{k=0}^{n-1} | F(\xi_k)-F(\xi_{k+1}) |^s }$, where the infimum is over all the $\delta$ partitions defined by the $n$th level subdivision of the closed interval $[0,\xi]\subset [a,b]$, and $F(\xi)=(\xi, f(\xi))$ is the graph of the parametrized curve. For a set with finite Hausdorff $s$-measure, one has ${\cal H}^s(F\cap [0,\xi])\propto \gamma^s(F,0,\xi)$. For a Cantor set $F$, one sets $f(\xi)=\xi^s, \ 0<s<1$.

In the following, we distinguish two fractal sets: (i) one denoted by  $F$ is a (fictitious) fractal set that must exist at the background and act as a sort of {\em seed}, for {\em a} specific choice of the duality structure and the corresponding valuation, but remains nascent at the present description, and (ii) the set denoted by $\Gamma_F$ that represents a  continuous fractal curve and one that is parametrized by the renormalized valuation $v(\xi)$. This fractal curve $\Gamma_F$ has the structure of a Lebesgue-Cantor type non-decreasing, singular or no-where differentiable curve and is generated by the Hausdorff measure of the seed fractal set $F$. In the context of the analysis on a fractal set, $\Gamma_F$ stands for the Cantor's singular staircase curve when the underlying set $F$ is a Cantor set \cite{dpr1, parv1}. For a Koch type set $F$, on the other hand, the curve $\Gamma_F$ denotes the graph of the corresponding Lebesgue Cantor type rise function \cite{parv2}.
}
\end{remark}

\begin{proposition}[{\bf Duality $\Rightarrow$ Fractality}] \label{dual}
A strictly dual valuation maps the extended null set ${\bf ^*O}\subset {\bf ^*R}$ into a segment of a fractal curve $\Gamma_F$ parametrized by a renormalized variable having the form of the associated devil's staircase function or, equivalently, the $s$-Hausdorff measure, $v(\xi)\propto {\cal H}^s(F\cap (0,\xi)), \ 0<\xi<<1$, where $s$ is the Hausdorff dimension of the seed fractal $F$. The deformed real line $\cal R$ is locally identical to the induced fractal curve $\Gamma_F$.
\end{proposition}

\begin{remark}
{\rm 
 Here we assume that the emergent fractal curve admits finite $s$-Hausdorff measure. In a more general situation, when Hausdorff measure is non-finite, i.e., either 0 or $\infty$, one should consider  generalized Hausdorff $h$-measures corresponding to a gauge function $h(x)$. 
}
\end{remark}

\begin{proof}
For a strictly dual asymptotic $x\geq \delta$, there exists $\tilde x, \ {\rm and \ a \ countable  \ sequence} \ \tilde x_i$, $\ i=1,2\ldots$ satisfying $0<\ldots<\tilde x_2<\tilde x_1<\tilde x<\delta\leq x$. The duality transformations satisfying the structure equation (\ref{ds2}) now generate another countable  sequence of asymptotics $x_i$ satisfying $\delta\leq \ldots<x_2<x_1<x$ so that the duality structure $\tilde x_ix_i=\lambda_i\delta^2$ and $\tilde v(\tilde x_i)v(x_i)=\mu_i$ is respected.  Consider countable, disjoint (clopen) covers $\tilde I^i$ and $I^i$ respectively for $(0,\delta)$ and $(\delta,x)$ such that $ \tilde I^i=(\tilde x_{i+1}, \tilde x_{i}]$ and $ I^i=(x_{i+1},  x_{i}]$. Notice that $\tilde v(\tilde x_2)>\tilde v(\tilde x_1) \ \Rightarrow \ v(x_2)<v(x_1)$.

Exploiting discrete ultrametricity, one assigns constant values to the renormalized valuations $\tilde v^i:=\tilde v^i(\tilde x_i)=\tilde v(\tilde I^i)$ and $v^i:=v^i(x_i)=v(I^i)$ which span the respective intervals $(0,\delta)$  and $(\delta,x)$ {\em continuously} (continuity follows from $v$ being a(n) (ulta)metric).  By a {\em judicious choice} of valuations $\tilde v^i$ for the cover $\tilde I^i$ and because of the duality structure, one can {\em always} make the sequence $v^i$ over the right hand covering sets $I^i$ {\em dense} and monotone increasing in $\bf R$.

Next,  these continuously spread constant, and densely defined, and increasing ultrametric values $v^i$ can be parametrized  by the scaling variable $\xi=\log x/\delta<1$, for a fixed choice of $x$ and $\delta$, and living in the interval  $(0,\xi)\subset \bf R$ ($\xi$ is a constant parameter relative to the covering intervals $I^i$ ), so that 
\begin{equation}\label{vmap}
v^i\mapsto v^i(\xi)=\xi^{{\alpha^i(\xi)}}, \ \alpha^i(\xi) <1 
\end{equation}
for each $i$, now define continuous mappings from $I^i\subset{\bf ^*R}$ to $\bf R$. The power law representation of the final equality is again dictated by the property of discrete valuation. Because of density of the sequence $v_i$, the limit map $v: (0,\xi)\mapsto \cal {R}$ defined by the sequence $v^i(\xi)$, as $i\rightarrow \infty$, is a continuous, Lebesgue-Cantor staircase type increasing function \cite{rudin, parv1, dpr1,dpr2,dss}, denoted above  by $S^s_{F}$ (remark (\ref{hm})) , corresponding to the asymptotically induced (seed) fractal set $F$, in the deformed neighborhood $\cal O$. Because of admissible choices in defining the maps $v^i(\xi)$, the induced continuous function $S^s_{F}$:= $v(\xi)=\xi^{{\alpha(\xi)}}$ can generally be increasing, differentiable almost everywhere, or  nowhere in $(0,\xi)\subset \bf R$, depending upon the particular application one is concerned with.
On the other hand, the mapping $\tilde v: \ {\bf ^*R}\ni (0, \delta) \mapsto  (\xi,\infty)  $ defined by the sequence $\tilde v^i(\xi)$, is also a continuous function, satisfying the duality condition $\tilde v(\xi)=\mu/v(\xi)=\mu\xi^{-{\alpha(\xi)}}$.

To justify the above assertion, one notes  that the renormalized valuation $v$, as represented above over the $\delta-$covers $I^i$ by a countable, densely defined parametrized values $v^i(\xi)$,  can always be arranged suitably to generate a {\em nontrivial}, and generally, {nonconstant, atomless measure} to the  $\delta-$covers  $\{J^i\}$  of the induced seed set $F$ in the form $v^i=k|J^i|^s, \ k>0$, so that projection of $J^i$ on  $(\delta,x)$ is $I^i$, and $s$ is the  Hausdorff dimension of the fractal set $F$ concerned. 

As a consequence, the finite, $s-$ dimensional Hausdorff measure of $F$ is now given by ${\cal H}^s(F)=\inf \sum|J^i|^s\propto\inf \sum |v^i(\xi)|, \ s>0$, in the limit $\delta\rightarrow 0$, when infimum is taken over all possible such $\delta$ covers. Because of monotonicity, this measure can now be inverted to define the  increasing continuous function by 
\begin{equation}\label{vh}
v(\xi)=\begin{cases}
{\cal H}^s(F\cap (0,\xi)),  \\
-{\cal H}^s(F\cap (-\xi,0))  
\end{cases}
\end{equation}
\end{proof}

\begin{remark}\label{dual1}
{\rm 
(i) For sufficiently small $v(\xi)<<1$, quadratic duality equation (\ref{dual}) implies that $v(\xi)\approx \mu/\kappa$ (neglecting $v^2$ compared to linear terms), so that $\tilde v(\tilde\xi)\approx \kappa$. For $v(\xi)\propto {\cal H}^s(F\cap \cal O)$, and $\mu$ a small constant, we get $v(\tilde\xi)\approx 1/{\cal H}^s(F\cap \cal O)$. 

In example (\ref{exdl}(ii)), the valuation $v_n(A_i^+)=i_n2^{-m_n}$ corresponds to the continuously distributed,  {\em dense} set of constant values assigned to the gaps of the middle third Cantor set $C_{1/3}$, and hence, in the limit $n\rightarrow \infty$, yields the Cantor's staircase function, that equals  $s=\log_{3} 2$ dimensional Hausdorff measure restricted on the Cantor set: $v(\xi)={\cal H}^s([0,\xi]\cap C_{1/3})$.

(ii) More generally, the quadratic duality equation (\ref{dl}) and its conjugate for $\tilde v$ admit arithmetical scaling solutions of corollary (\ref{wsdn}), as well as non-scaling solutions by the quadratic equation (\ref{dl}). In fact, the above two limits, viz, $\kappa>> \mu$ and $\kappa\sim O(\mu)$ exhaust the set of nontrivial solutions, since the limit $\kappa<<\mu$ reduces to the self dual limit. A detailed analysis of weakly dual arithmetical asymptotics would be treated else where.}
\end{remark}

\begin{example}[\bf Strictly dual] {\rm 
To give an explicit construction to exemplify the generation of $\Gamma_F$, we first give a general construction of a self similar set $F$ that could be facilitated by a nontrivial duality structure  that creates a finite region even in the asymptotic limit $x\rightarrow 0$ that could support a countable set of densely defined nontrivial valuations leading to a fractal curve.

We begin by considering a dense  sequence of asymptotics $x_i$ satisfying $\delta\leq \ldots<x_2<x_1<x$ in the interval $[\delta,x)$. Such a choice exists, by nontrivial duality structure. For the sake of convenience, denote the corresponding dense sequence in logarithmically rescaled interval $[0,\xi)\subset \bf R$ by the usual ordering $0<\xi_1<\xi_2<\ldots$. Let us now choose a finite subsequence $\xi_{i_k}$ such that $\xi_{i_0}=0, \ \xi_{i_{k}}<\xi_{i_{k+1}}$ and $k=0,1,\ldots, m^i-1$ for some positive integer $m$. Define $f_i: (0,\xi) \rightarrow {\bf R}^+$ by assigning constant values {\em continuously} across the closed intervals $[\xi_{i_k},\xi_{i_{k+1}}]$, which are, however, parametrized by the parameter $\xi$. To see explicitly, how the assigned parametrized values behave, let us choose, without loss of generality, $f_0(\xi)=\xi$, and then assume that at each subsequent stages as $i$ runs over positive integers $i=1,2,\ldots$, $f_0$ is acted  upon by a nonlinear generating operation $\cal N$ such that $f_i(\xi)={\cal N}^i(f_0(\xi)), \ {\cal N}^0=Id$. In general, the nonlinear generator $\cal N$ is defined  by an iterated function systems (IFS), having $m$ distinct components. For completeness, we give two definitions of $\cal N$ via well known IFS. Assuming $f_0(\xi)$ as the initiator for a self similar fractal curve, and the generating operator $\cal N$ the corresponding generator, one can now convince easily \cite{falc, hutch, oka} that the limiting curve $f(\xi)=\underset{i\rightarrow \infty}\lim f_i(\xi)$ does converge uniformly to a continuous fractal curve of certain Hausdorff dimension $s\geq 1$. The nonlinear operator thus defines a continuous deformation in the space of real valued continuous mappings.

When the limiting curve $f$ is nondecreasing, and singular (in the sense of Lebesgue measure), we are done, the graph of the curve $(\xi,f(\xi))$ being the deformed fractal curve $\Gamma_F$. However, when $f$ is no-where differentiable self similar curve $F$ that is induced at the background of the concerned deformation process, $\Gamma_F$ is the graph of the strictly increasing function $v(\xi)$, that is given by the associated Hausdorff measure  $v(\xi)={\cal H}^s(F\cap (0,\xi))$.

 (i) A family of Koch curves can be defined by the iterated function system (IFS) of contracting similarities $S =\{S_1,S_2,S_3, S_4\}$ acting on $I=[0,1]\subset C$, the complex plane is defined by 
\begin{equation}\label{koch}
S_1(z)= Lz, \ S_2(z)=L(1+az), \ S_3(z)=L(1+a+\bar az), \ S_4(z)= L(1+2\cos\alpha+az))
\end{equation}
where $a=e^{i\alpha}$, $L=(2+2\cos\alpha)^{-1}$, and $z\in C$ \cite{koch}. The nonlinear operator $\cal N$ is here identified with 
the  Hutchinson operator \cite{falc, hutch} associated to the given IFS: ${\cal N}(A) = \bigcup^4_{i=1} S_i(A), \ A\subset \bf R$. 

(ii) As a second example, we may recall Salem's construction of a parametrized family of singular functions $f_a(x)$ defined and analyzed  in Ref.\cite{salem, oka}. The function $f_a(x)$ is no-where differentiable if $2/3\leq a<1$, nondifferentiable almost every where for $a_0< a<2/3$ and differentiable almost every where if $0<a<a_0$, where $a_0$ is the unique real root of $54a^3-27a^2=1$ \cite{oka}. In particular, $f_{1/3}(x)=x$ and $f_{1/2}(x)$ is the middle third Cantor function. The nonlinear generator $\cal N$ again is defined by the corresponding Hutchinson operator \cite{hutch}. 
}
\end{example}

\subsection{Renormalization Group}
The above duality induced nonlinear transformation in the asymptotic neighbourhood of 0, now lends itself to an interesting interpretation of renormalization group (RG). Intermediate asymptotics\cite{baren} and renormalization group (RG) (actually a semi-group)\cite{ono, golden1} have found significant applications in scaling phenomena\cite{golden2} that arise generally in an intermediate stage (mesoscopic scale) of diverging asymptotic of nonlinear phenomena or processes.

We note first of all that our definitions of renormalized asymptotics (\ref{renom}) are actually intermediate asymptotic quantities, since these are defined and meaningful in an  intermediate  region $0<< \delta<x<<1$. To see explicitly the relationship with RG and the present approach, thus justifying the nomenclature used, let the transformation  of a null asymptotic $x$ to a finite renormalized valuation $v(x)$ be denoted by the RG operator ${\cal T}: \ {\cal T}(x)=\frac{\log x/\delta}{\log \delta^{-1}},$ in the limit $x\rightarrow 0, \ \delta\rightarrow 0$, satisfying $ \ 0<\delta<x$. Clearly $v(x)=|{\cal T}(x)|$. Since, $v(x)\in \bf R$, the composition of $\cal T$ is well defined, and hence it generates a semi-group.

Next, we derive the RG flow equations \cite{ono} in the context of renormalized asymptotics $X^+$. To begin with, let us reinterpret the original asymptotic parameter $x$ as   primary scale relative to which the renormalized asymptotic $X^+$ has the representation $X^+=x\cdot x^{-v^+(X^+)}(1+o(1))$. Clearly, the rescaled quantity $\bar X^+=X^+/x$ is by definition a divergent quantity in the limit $x\rightarrow 0$. Introducing a secondary running phenomenological scale $\delta$, one could rearrange the divergent quantity $\bar X^+$ into a "phenomenologically" finite quantity $\bar X^+_{ph}$ in the form \cite{ono}
\begin{equation}
\bar X^+_{ph}=Z\bar X^+=\xi^{v^+(X^+)}, 
\end{equation}
when one introduces the rescaled variable $\xi=\delta/x<1$ and the renormalization constant $Z=\delta^{v^+(X^+)}$, to eliminate (absorb) the divergent part in the asymptotic $X^+$. Clearly, the rescaled renormalized asymptotic $\bar X^+$ should be independent of the intermediate phenomenological scale $\delta$ so that the renormalization equation  $\frac{\partial \bar X^+}{\partial\ln\delta}=0$ is satisfied,   which then translates into 
\begin{equation}\label{cse1}
\frac{\partial Z^{-1}}{\partial \ln \delta}\bar X^+_{ph}+Z^{-1}\frac{\partial \bar X^+_{ph}}{\partial \ln \delta}=0
\end{equation}
a Callan-Symanzik type RG equation \cite{ono, golden2}. Noting that $\frac{\partial \ln Z}{\partial\ln \delta}=v^+$ and using the substitution $\delta=\xi x$ for a {\em fixed} $x$ in the derivation in the second term, one obtains $\bar X^+_{ph}\propto \xi^{v^+(\xi)}$, as above. As a consequence,  the valuation here has the status of an anomalous scaling dimension of an underlying dynamical problem.

In the present formalism, one is essentially interested in the renormalized valuation $v^+(X^+(x))$, which is directly introduced through the definition (\ref{vnorm}), bypassing the RG analysis outlined above. Moreover, interpreting the renormalized valuation $v^+(x)$ as the renormalized, running coupling  constant of field theory \cite{golden2}, one, however, notes that $\frac{\partial\bar X^+_{ph}}{\partial \ln x}\neq 0$. As a consequence, one sets 
\begin{equation}\label{beta}
\frac{\partial\bar v^+(x)}{\partial \ln x}=\beta(v^+(x))
\end{equation}
where $v^+(x)\equiv v^+(\xi)$, and $\beta(v^+)$ is the RG beta function \cite{golden2} for the corresponding RG flow of the nontrivial valuation $v^+$. One justifies this $\beta$- function interpretation in the following manner.

For a fixed choice of $x$ and $\delta$, we can rewrite $\log x/\delta=\tilde \xi=\eta\log x^{-1}$ for a new rescaled parameter $\eta$. For a fixed value of $x$, the mapping $v:I\rightarrow I, \ I=[0,\xi]$, acts as  an anomalous dimension of the underlying fractal set. However, as $x\rightarrow 0$,  $v^+=v(\eta\log x^{-1})$, for a fixed $\eta$, now defines a one parameter family of continuous deformations in the space of continuous mappings  that deforms an initial self dual $v_s=1$ (say), to the final limiting fractal $v_f$, as $x$ varies from an initial self dual asymptotic of the form $x_s=e^{-k}<<1, \ k>0$ large but fixed, for instance, to a limiting $x_f=e^{-u}\rightarrow 0, $ for $u\rightarrow \infty$. One needs to recall here that $v_s(\xi)=v_s(k\eta)=v_s(\eta)$ by scaling invariance, proposition (\ref{invprop}).

As a consequence, as $x\rightarrow 0$, the RG flow of $v^+$, in the space of continuous functions, is controlled by a specific $\beta$-function that arises naturally in the context of a  nonlinear    generating operator $\cal N$ that was introduced in proposition (\ref{dual}) to generate a self similar fractal set in the interval $(0,\xi)$. 

For the sake of clarity, we next re-derive the above inferences more formally. Recall that the rescaled renormalized asymptotic, according to the definition (\ref{renom}), may be written as $\bar X^+=\bar X^+(\delta, v^+(x))$, being essentially a function of two scales $x$ and $\delta$, where $\delta$ is interpreted as the phenomenological scale in the sense of RG, $x$ appear through the valuation $v^+(x)$. In the framework of the conventional dimensional analysis, $\bar X^+$ is, however, trivially a constant. In the present framework, we, however, wish to endow this quantity a nontrivial {\em phenomenological} value, by bringing in a renormalization constant $Z$: $\bar X^+_{ph}=Z\bar X^+$, so that $\bar X^+_{ph}$ is nontrivial and finite. The independence of $\bar X^+$ on the phenomenological scale now translates into the full Callan-Symanzik RG equation \cite{golden2}
\begin{equation}\label{cse2}
\frac{\partial \bar X^+_{ph}}{\partial \ln \delta}+\beta(v^+)\frac{\partial \bar X^+_{ph}}{\partial v^+}+\gamma(\delta)\bar X^+_{ph}=0
\end{equation}
where $\frac{\partial\bar v^+}{\partial \ln \delta}=\beta(v^+)$ is the {\em beta function} and  $\gamma(\delta)=\frac{\partial \ln Z}{\partial \ln \delta}$ is an {\em anomalous dimension}. Solving this first order partial differential equation along the characteristic defined by the beta equation, one immediately recovers $\bar X^+_{ph}$ as $\bar X^+_{ph}\propto e^{-\int \gamma(\delta) d\ln \delta}$. For a fixed $x$, one now identifies $v^+$ with the anomalous dimension through the equality $v^+(\delta)\ln \delta=\int \gamma(\delta) d\ln \delta$. For a general $x$, the beta function then controls the variability of $v^+$. 

To summarize, we have presented two distinct RG interpretations of the renormalized valuation $v^+$: in one hand $v^+$ has the status of anomalous dimension in the sense that an asymptotic $x\in \bf R$ acquires in $\cal R$ an extra scaling exponent in the form $x^{1-v^+}$, so that $v^+$ does have the status of an anomalous fractal exponent. On the other hand the flow of $v^+$ in the function space also has a RG beta function interpretation.   

\section{Applications: Wave Equations}

\subsection{Renormalizability of Derivative}

Consider a continuously differentiable function $f$ in the interval [0,1]. We now show that  the derivative of $f$ at $a\in [0,1]$, $f^{\prime}(a)=\underset{\Delta x\rightarrow 0}\lim \frac{\Delta f}{\Delta x}$, where $\Delta f=f(x)-f(a)$ and $\Delta x=x-a$, is {\em invariant under renormalizability transformation $x\mapsto v(x)$}. To prove this, we first reinterprete the limit $x\rightarrow a$ as $v(\Delta x):=\Delta v(x)=v(x)\rightarrow 0$ (final equality follows by translation invariance). Let us also make use of the definition $v(f(x):=f(v(x))$, for the self similar extension of $f$ in the extended neighbourhood of 0 in $\bf {^*R}$, that is in $\bf O$ \cite{dss}. As a consequence, we have $v(\Delta f)(x)=\Delta f(v)(x)$. Notice that, here we are assuming that both the renormalized extensions of $x$ and $f$ are parametrized by the identical scale $\delta$.

Next, rescale the numerator and denominator of the differential quotient $\frac{\Delta f}{\Delta x}$ by the scale $\delta$ and use the translation invariance, as well as the first order approximation $\log x/\delta\approx x/\delta - 1$ etc, and then make yet another rescaling by logarithmic quantity $\log\delta^{-1}$, and finally take the usual limit $\Delta x\rightarrow 0$, together with the scaling limit $\delta\rightarrow 0$, successively, to get 
$$
\underset{\Delta x\rightarrow 0}\lim \frac{\Delta f}{\Delta x} \ \mapsto \ \underset{\Delta v(x)\rightarrow 0}\lim \frac{\Delta f(v)}{\Delta v(x)}
$$
Consequently, the derivative of a continuously differentiable function $f$ at any point $a$ is invariant under the renormalization transformation, so that $f^\prime(a)\mapsto \ D_vf(x)=\frac{d f(v(x))}{d v(x)}|_{x\in \bf O}$.

For a nondifferentiable, or singular function $f(x)$, with $x$ varying on a fractal curve $\gamma$, we now define the duality induced asymptotic derivative by extending naturally the above renomalizability invariance. Let $x\in \gamma$, a fractal curve. One invokes appropriate renormalizability, so as to map the asymptotic $x \ (\rightarrow 0)$ to an effective renormalized variable $v(x)$ in the form of the associated integral (devil's) staircase function \cite{parv2}, that is  monotonic increasing. Then the asymptotic derivative on the fractal curve is defined by 
\begin{equation}\label{fracderv}
D^a f(x):= \frac{d^{\gamma}f(x)}{d x^{\gamma}}:= \underset{\Delta v(x)\rightarrow 0}\lim \frac{\Delta f(v)}{\Delta v(x)}
\end{equation}
when the right hand limit exists, and where $d^{\gamma} x = dv(x)$. This definition clearly parallels the one introduced in \cite{parv1} and in subsequent applications in local fractal or fractional derivatives \cite{kol1, bel1,bel2}. As a consequence the formalism of local fractional type calculus developed in the recent literature \cite{yang1,yang2} could be easily incorporated into the present context. One should, however, realize the basic difference between the present formalism  with those in the literature. As remarked in Introduction, the present approach develops a larger framework to address the problem of continuous deformation of the ordinary integral order calculus and dynamics into  a fractional, and more generally, a fractal calculus and dynamics, in an asymptotically late time or in an asymptotically small spatial extensions. The fractional or local fractional techniques available in literature, however, at best, are realized in an ad hoc manner, and one is unable to bridge the gulf between integral and non-integral  dimensions continuously, as is being proposed here.

To exemplify the continuous deformation of the dynamics, let now consider the linear wave equations in one and two dimensions successively.

\subsection{One Dimension}

Let us consider one dimensional linear wave equation
\begin{equation}\label{1d}
\frac{\partial^2 U}{\partial t^2}=c^2  \frac{\partial^2 U}{\partial x^2}, \ t>0, \ 0< x< l
\end{equation}
along with homogeneous Dirichlet boundary conditions (BC): $U(t,0)=U(t,l)=0$ and initial conditions (IC): $U(0,x)=h(x), \ U_t(0,x)=0, \ 0\leq x\leq l$. Assuming $h(x)$ satisfy suitable Dirichlet's conditions, the equation admits the well known double Fourier series representation of the classical (smooth) solution in  the time $t$ and space variable $x$. 

In a practical application, initial deflection more often be an irregular fractal type curve rather than simply piecewise smooth.  In such a vibration of a fractal string, appropriate spatial variable would not be the ordinary real variable $x$, but, instead, an appropriately renormalized variable $v(x)$, that parametrizes the fractal string $\gamma$ naturally in the form given by the concomitant fractal integral staircase function $v(x):=S_{\gamma}^s (x)$ \cite{parv2}, where $s$ being the Hausdorff dimension of the fractal string. The associated vibration is then denoted by $U_{\gamma}(t,x):=U(t,v(x))$, and the corresponding continuously deformed wave equation on the fractal string is given by  
\begin{equation}\label{1dfrac}
\frac{\partial^2 U}{\partial T^2}=v(c)^2\frac{{\partial^{2\gamma}} U}{{\partial} x^{2\gamma}}, \ t>0, \ 0< v(x)< v(l)
\end{equation}
where $v(c)$ denotes the dimensionless wave velocity in  the fractal string dimension, and the rescaled time variable $T=ct$ has the dimension of space. The boundary and initial conditions are also transformed as $U(T,0)=U(T,v(l))=0 $ and $U(0,v(x))=h(v(x)), \ U_t(0,v(x))=0$ respectively. 

A remark is in order here. The equation (\ref{1dfrac}) is assumed to be a model of fractal string vibration when the entire  formalism of the ordinary real calculus is transported over the fractal space $\Gamma$ (in virtue of conjugacy between ordinary and fractal calculus \cite{parv2}), represented here by the fractal curve $\gamma\subset \Gamma$ of a given Hausdorff dimension $s$ and equipped with a renormalized effective variable $v(x)$. This effective variable should be treated as the basic (independent) variable in the calculus on the fractal space. Letting, at first, the time parameter $t$  unaltered, all the relevant physical (dynamical) quantities in the wave propagation problem, viz, fractal wave velocity $c_f$, fractal wave number $k_f$ etc must have the structure determined uniquely by the renormalization prescription, concomitant to the given fractal medium, so that one must have $c_f=v(c)$, $k_f=v(k)$, etc, where $c, \ k$ denote the corresponding parameters in the conventional (non-fractal) wave equation in $\bf R$. It follows that as a function of $k$, the corresponding fractal wave number $v(k)$ has the structure of the associated devil's staircase corresponding to the underlying fractal curve and so on etc. 

Incorporating the separation of variables and Fourier series theorems as formulated on a fractional space \cite{yang2}, the infinite series representation of the fractal vibration can be written down as 
\begin{equation}\label{fracsol1}
U_{\gamma}(t,x)=\sum_1^{\infty} a_n \cos \omega_f T\sin k_f v(x)
\end{equation}
where the fractal Fourier coefficient $a_n=\frac{2}{v(l)}\int_0^{v(l)} h(v(x)\sin  k_f v(x) d^{\gamma}x$. Consequently, the fractal wave dispersion relation is obtained as $\omega_f=v(c)k_f$, where the fractal wave number is determined as discrete eigenvalues $k_f:=v(k)=\frac{n\pi}{v(l)}$. It follows that the fractal wave frequency $\omega_f=v(c)v(k)$, as a function of non-fractal wave number $k$, has the form of the devil's staircase corresponding to the fractal string, and so is nowhere differentiable continuous function of $k$. The similar results are  also reported in the literature \cite{fracfreq}. 

One can now proceed one step further in modeling more precisely the fractal wave vibration. Recognizing the fractal nature of the wave frequency  $\omega_f$, it becomes imperative, because of dimensional reason,  to fractalize the time axis $T$ as well, so that the associated fractal time variable, in the wave equation (\ref{1dfrac}), should be raised to the fractal time  as $v(T)$, in conformity with the spatial fractal structure $v(x)$ and the fractal frequency $\omega_f$.  This transformation in time variable is of course admissible in the present formalism because of the inherent renormalizability invariance of the wave equation (\ref{1d}) under the asymptotic renormalization transformations $x\mapsto v(x), \ T\mapsto v(T)$, as proved in Sec 3.1. This also constitutes an example of {\em coherent and correlated spatio-temporal duality } introduced in Ref.\cite{dss}. As a consequence, more accurate model of a fractal string vibration seems to be 
\begin{equation}\label{1dfrac2}
\frac{{\partial^{2\gamma}} U}{\partial T^{2\gamma}}=v(c)^2\frac{{\partial^{2\gamma}} U}{{\partial} x^{2\gamma}}, \ t>0, \ 0< v(x)< v(l)
\end{equation}
that is consistent with the asymptotic scaling dimensions and the fractal dispersion law. The associated fractal vibration is now given by (\ref{fracsol1}), but the time variable $T$ replaced by the corresponding fractal time $v(T)$. It is of interest to notice that the time parameter in a dynamical problem in Euclidean space is assumed to be a monotonic increasing real variable, by default. In a fractal space dynamics, the corresponding time parameter ought to have a fractal behaviour, so as to make the underlying  equation and relevant dynamical quantities (fractal) dimensionally correct. Interestingly, this principle of spatio-temporal coherence in duality structure parallels quantization principle in quantum mechanics, when both space and time quantizations are implemented simultaneously to describe a microscopic quantal state.

\subsection{Two Dimension}

Consider the two dimensional wave equation in a bounded domain $D$ in plane with boundary $\Gamma=\partial D$ given by 
\begin{equation}\label{2dwave}
\frac{\partial^2 U}{\partial t^2}=c^2 ( \frac{\partial^2 U}{\partial x^2} +\frac{\partial^2 U}{\partial y^2})
\end{equation}
with initial and boundary conditions respectively as $U(0,x,y)=h(x,y), \ U_t(0,x,y)=0, \ (x,y)\in D$ and $U(t,x,y)|_{\Gamma}=0, \ t\geq 0$. Recently, Ma and Su \cite{su} presented an infinite lacunary solution of the above wave equation for dyadic (quadratic type 2) Koch type curve $\Gamma$. Their method is simple, based on the standard separation of variable method over a sequence of approximating regions $D_k$ with prefractal polygonal type boundaries $\Gamma_k=\partial D_k$ for $k=1,2,\ldots$. This approximating regions are supposed to converge to the original region $D$ with fractal boundary $\Gamma$. The sequence of approximating boundaries $\Gamma_k$ considered by the authors are a sequence of squares circumscribing the quadratic koch curve form above and are given by $\Gamma_k: -a_{2k}\leq x\leq a_{2k+1}, \ -a_{2k}\leq y\leq a_{2k+1}$, where the sequence $a_k$ is defined by the recurrence, $4a_{2k}=a_{2k-1}, \ a_{2k+1}=1+a_{2k},  \ a_{-1}=1, \ k=0,1,2,\ldots$, so that $a_k\rightarrow 1/3, \ k\rightarrow \infty$. The suggested  lacunary series solution $U(t,x,y)$ of the problem was reported as \cite{su}
\begin{align}\label{2dsol}
U(t,x,y) &=\underset{k\rightarrow \infty}\lim U_k(t,x,y), \ {\rm where} \\
U_k(t,x,y) &= 
\sum_m\sum_n A_{k,m,n}\cos 4^k\pi \sqrt{(m^2+n^2)}t \times 
\sin 4^k\pi m x \times \sin 4^k \pi n y
\end{align}
Here, the Fourier coefficients $A_{k,n,m}$ are given by 

$$A_{k,m,n}  =  4{\int\int}_{D_k} h_k(x,y)\sin 4^k\pi m x \times \sin 4^k \pi n y  dx dy$$
 and $h_k(x,y)$ is the boundary condition on the   $k$th approximating boundary $\Gamma_k$. Under suitable conditions on the boundary functions, the above solution was shown to exists uniquely as a continuous function in the domain $D$. However, the lacunary series in (\ref{2dsol}) turns out to be no-where differentiable in the domain $D$ with fractal boundary $\Gamma$. This is justified in \cite{su} as indicative of the failure of the classical wave equation in the limit $k\rightarrow \infty$, when the fractal boundary arises as the limit set. Notice that, the approximating solutions $U_k(t,x,y)$ do exist as the classical solution of the wave equation with approximating boundaries $\Gamma_k$ under appropriate conditions on the boundary data $h_k(x,y)$. However, {\em this interpretation is unsatisfactory, as the iterated boundaries $\Gamma_k$ do not converge to the intended fractal boundary $\Gamma$, but rather to the limit square $\Gamma^{\prime}: -1/3\leq x\leq 4/3,\ -1/3\leq y\leq 4/3$ circumscribing the intended fractal $\Gamma$}. 

In the present approach, the duality structure provides one with an interesting avenue, not only in removing the above deficiency, but also to offer a more realistic modeling of the two dimensional wave vibration by a fractal wave equation in appropriate fractal scaling variables.  Introduce the scale $\delta$ as $\delta=1/4^k$, invoking the duality structure in the 2D wave problem in an asymptotic neighbourhood of a point $(t_0,x_0,y_0)\in {\bf R}\times D$. Next, define the renomalization transformations,  both in dimensionless space and time variables $x,\ y$ and $T=ct$, so that $x\mapsto v(x)$, $y\mapsto v(y)$ and $T\mapsto v(T)$, where $v$ now corresponds to the devil's staircase for the given quadratic Koch curve in respective  variables, as indicated in Sec.3.1. Translation invariance and the renormalizbility invariance of the 2D wave equation now recast (\ref{2dwave}) into a fractal equation
\begin{equation}\label{2dfrac}
\frac{{\partial^{2\gamma}} U}{\partial T^{2\gamma}}=v(c)^2(\frac{{\partial^{2\gamma}} U}{{\partial} x^{2\gamma}} +\frac{{\partial^{2\gamma}} U}{{\partial} y^{2\gamma}}), \ t>0, \ (x,y)\in D, \ \partial D=\Gamma
\end{equation}
where $d^{\gamma}T=dv(T), \ d^{\gamma}x=dv(x)$ and $d^{\gamma}y=dv(y)$ are the relevant fractal differentials. To illustrate the aforesaid renormalized scale transformations more explicitly, one first shifts $x\mapsto \Delta x=x-x_0$, by translation invariance, and then exploits the limiting rescaling symmetry of the form $v(x)=v(\Delta x)\approx\frac{\log 4^k\Delta x}{k\log 4}\approx \frac{4^k\Delta x-1}{k\log 4}$ etc of each term of the wave equation (\ref{2dwave}), following Sec.3.1, and finally let $k\rightarrow \infty$, to reproduce fractal wave equation (\ref{2dfrac}). The fractor $v(c)^2$, however, is introduced by hand from dimensional analysis. In formulating the above fractal equation we have also invoked the principle of {\em coherent and correlated spatio-temporal duality structure} as explained in Sec.3.2 (final paragraph).

We remark that the  fractal equation (\ref{2dfrac}) is realized as the limiting equation from the renormalization scale transformations in the limit $k\rightarrow \infty$. For any finite $k$, the classical wave problem with non-fractal boundary $\Gamma_k$ does admit a classical smooth solution. The limit $k\rightarrow \infty$ is meaningless in the classical sense. However, under the duality enhanced renormalizability transformations, the original wave equation is transformed into the fractal wave equation (\ref{2dfrac}), that admits a continuous solution 
\begin{equation}\label{2dsmth}
U_{\gamma}(t,x,y) = 
\sum_m\sum_n A_{m,n}\cos \pi \sqrt{(m^2+n^2)}v(T) \times 
\sin \pi m v(x) \sin \pi n v(y)
\end{equation}
in the appropriate fractal continuous variables $v(T), \ v(x)$ and $v(y)$ and where $A_{m,n}  =  4{\int\int}_{D} h(x,y)\sin \pi m v( x) \times \sin  \pi n v(y) \ d^{\gamma}xd^{\gamma}y$. Moreover, this solution is well defined and smooth in the extended notion of fractal differentiability that is being developed here and in Ref.\cite{dss}. However, the graphs of the fractal wave solution would reflect the fractal characteristics that are inherited  from the fractal boundary by the asymptotic fractal variables.

\section{Closing remarks}

We have presented in this paper a novel analytic framework to address the interesting question of a possible continuous deformation of the ordinary integral order dynamics into a local fractional type dynamics. The intended continuous deformation is facilitated by  formulating the concept of duality structure that is shown to exist in the family of null or divergent asymptotics (sequences). The significance of self dual and strictly dual asymptotics are discussed in  offering a new interpretation of emergent smooth or fractal like nonsmooth (chaotic) structure in a nonlinear dynamical problem. A renormalization group (RG) interpretation of the emergent nonlinear structures is also presented. Two applications of the formalism are considered in the context of one and two dimensional linear wave equations. The RG induced continuous deformation transform the wave equations into fractal (local fractional) wave equations modeling respectively fractal string vibration or standing wave vibration with fractal boundary. Applications to a more realistic nonlinear problem will be considered in future.

{\bf Acknowledgement:}
The first author (DPD) thanks IUCAA, Pune for awarding a Visiting Associateship.

\end{document}